\documentclass{article}

\usepackage[utf8]{inputenc}
\usepackage{booktabs}
\usepackage{color}
\newcommand{\href}[1]{#1} 

\usepackage{hyperref}
\usepackage{ifthen}
\usepackage{geometry}
\usepackage[all]{xy}
\xyoption{arc}
\usepackage{makeidx}
\usepackage{amsmath,amssymb,amstext,latexsym,amsthm,amscd} 
\newtheorem{thm}{Theorem}[section]
\newtheorem*{thm*}{Theorem}
\newtheorem{lemma}[thm]{Lemma}
\newtheorem*{lemma*}{Lemma}
\newtheorem{prop}[thm]{Proposition}
\newtheorem*{prop*}{Proposition}

\newtheorem{conjecture}[thm]{Conjecture}

\theoremstyle{definition}

\newtheorem{rmk}[thm]{Remark}

\theoremstyle{remark}

\newcommand{\bbQ}{\mathbb{Q}}

\newcommand{\bbF}{\mathbb{F}}
\newcommand{\bbC}{\mathbb{C}}
\newcommand{\bbZ}{\mathbb{Z}}
\newcommand{\bbP}{\mathbb{P}}

\newcommand{\bbR}{\mathbb{R}}

\newcommand{\QQ}{\bbQ}
\newcommand{\RR}{\bbR}

\newcommand{\CC}{\bbC}
\newcommand{\FF}{\bbF}

\newcommand{\ZZ}{\bbZ}

\newcommand{\PP}{\bbP}

\newcommand{\cD}{\mathcal{D}}
\newcommand{\cO}{\mathcal{O}}

\newcommand{\cH}{\mathcal{H}}

\newcommand{\cM}{\mathcal{M}}

\newcommand{\cE}{\mathcal{E}}

\newcommand{\cT}{\mathcal{T}}

\newcommand{\cV}{\mathcal{V}}

\newcommand{\fp}{\mathfrak p}
\newcommand{\fq}{\mathfrak q}

\DeclareMathOperator{\GL}{GL}

\DeclareMathOperator{\Div}{Div}

\DeclareMathOperator{\res}{res}

\DeclareMathOperator{\dlog}{dlog}

\DeclareMathOperator{\SL}{SL}

\DeclareMathOperator{\red}{red}

\DeclareMathOperator{\Ind}{Ind}

\newcommand{\injects}{\hookrightarrow}

\newcommand{\smtx}[4]{\left(\begin{smallmatrix}#1&#2\\#3&#4\end{smallmatrix}\right)}

\newcommand{\emphh}[2][ ]{%
\ifthenelse{\equal{#1}{ }}{\index{default}{#2}{\emph{#2}}}{\index{default}{#1@#2}{\emph{#2}}}%
}

\newcommand{\emphhh}[3][ ]{%
\ifthenelse{\equal{#1}{ }}{\index{default}{#2}{\emph{#3}}}{\index{default}{#1@#2}{\emph{#3}}}%
}

\def\sumprime{\mathop{\sum{\raise3pt\hbox{${}'$}}}} 
\def\prodprime{\mathop{\prod{\raise3pt\hbox{${}'$}}}}

\newcommand{\eps}{\varepsilon}
\renewcommand{\epsilon}{\varepsilon}

\newcommand{\tto}[1]{%
\ifthenelse{\equal{#1}{}}{\to}{\stackrel{#1}{\longrightarrow}}}

\title{A quaternionic construction of $p$-adic singular moduli}
\author{Xavier Guitart, Marc Masdeu, Xavier Xarles}
\begin{document}
\maketitle
\renewcommand{\emphh}[1]{{\emph{#1}{}}}

\newcommand{\Sh}{\operatorname{Sh}}
\newcommand{\fl}{\mathfrak{l}}
\newcommand{\BS}{\operatorname{BS}}
\newcommand{\MS}{\operatorname{MS}}
\newcommand{\Z}{\mathbb{Z}}
\newcommand{\ra}{\rightarrow}
\newcommand{\lra}{\longrightarrow}
\renewcommand{\div}{\operatorname{div}}
\renewcommand{\red}{\operatorname{red}}

\newenvironment{mytable}[2]
{\begin{tabular}{|*{#1}{l|}}
\toprule
\multicolumn{#1}{c}{#2}\\
\midrule}
{\bottomrule\end{tabular}}

\begin{abstract}
 Rigid meromorphic cocycles were introduced by  Darmon and Vonk as a conjectural $p$-adic extension of the theory of singular moduli to real quadratic base fields. They are certain cohomology classes of $\SL_2(\ZZ[1/p])$ which can be evaluated at real quadratic irrationalities and the values thus obtained are conjectured to lie in algebraic extensions of the base field. In this article we present a similar construction of cohomology casses in which $\SL_2(\ZZ[1/p])$ is replaced by an order in an indefinite quaternion algebra over a totally real number field $F$. These quaternionic cohomology classes can be evaluated at elements in almost totally complex extensions $K$ of $F$, and we conjecture that the corresponding values lie in algebraic extensions of $K$. We also report on extensive numerical evidence for this algebraicity conjecture.
\end{abstract}
\section{Introduction}

Classical singular moduli are the values of the $j$-function at imaginary quadratic arguments in the complex upper half plane. They turn out to be algebraic numbers, and in fact they play a central role in explicit class field theory because they generate the ring class fields of imaginary quadratic fields. While this theory has been generalized to some extent for CM fields, no satisfactory analogues of the $j$ function are known for other types of fields. Even for real quadratic fields, the finding of meromorphic functions whose special values generate their class fields remains an important open problem.

A recent breakthrough in this direction is the $p$-adic approach proposed by Darmon--Vonk in \cite{darmon-vonk}. The $j$-function is a meromorphic function on the complex upper half plane $\cH$, and therefore it cannot be evaluated at real quadratic irrationalities. A first key idea of Darmon--Vonk is to replace $\cH$ by the $p$-adic upper half plane $\cH_p = \CC_p\setminus \QQ_p$, which does contain plenty of real quadratic irrationals (also called RM points), and to consider rigid meromorphic functions on $\cH_p$. A second insight comes from the observation that, although the classical $j$-function is invariant under the action of $\SL_2(\Z)$ on $\cH$, in the $p$-adic setting is more convenient to consider the $p$-arithmetic group $\SL_2(\Z[1/p])$ acting on $\cH_p$. It turns out that the only $\SL_2(\Z[1/p])$-invariant functions on $\cH_p$ are constants; that is, if $\cM^\times$ denotes the multiplicative group of rigid meromorphic functions on $\cH_p$ then 
\begin{align*}
  H^0(\SL_2(\Z[1/p]),\cM^\times) = \CC_p^\times,
\end{align*}
so no interesting functions arise in this way. This motivates considering $H^1$ instead of $H^0$, and leads to the key notion of rigid meromorphic cocycles introduced in \cite{darmon-vonk}, which are elements of
\begin{align*}
  H^1(\SL_2(\Z[1/p]),\cM^\times)
\end{align*}
satisfying the extra condition of having a representative whose restriction to the subgroup of upper triangular matrices is constant. Even though rigid meromorphic cocycles are not functions but cocycles with values in functions, there is a sensible notion of evaluating a rigid meromorphic cocycle $J$ at an RM point $\tau \in K\cap \cH_p$ (here $K$ here denotes the real quadratic field in which $\tau$ lies). The main conjecture of \cite{darmon-vonk} predicts that this resulting $p$-adic number $J[\tau]$ is, if fact, an algebraic number that lies in a  compositum of ring class fields of real quadratic fields. 

The main evidence for the validity of this conjecture is experimental. Indeed, an important feature of \cite{darmon-vonk} is the construction of certain explicit rigid meromorphic cocycles in the case where $p$ is a monstruous prime\footnote{i.e., when $p$ belongs to the set $\{2, 3, 5, 7, 11, 13, 17, 19, 23, 29, 31, 41, 47, 59, 71 \}$}. For such a prime $p$, Darmon--Vonk construct a rigid meromorphic cocycle $J^+_{\tau_1}$ associated to any RM point $\tau_1\in K_1\cap \cH_p$, given as an explicit infinite product of rational functions with zeroes and poles supported in the $\SL_2(\ZZ[1/p])$-orbit of $\tau_1$. They compute, to high $p$-adic accuracy, the values $J^+_{\tau_1}[\tau_2]$ of this cocycle evaluated at other RM points $\tau_2\in K_2\cap \cH_p$, and verify in many examples that these values are $p$-adically close to algebraic numbers lying in the predicted field extension, which is the compositum of ring class fields of the quadratic orders associated to $\tau_1$ and $\tau_2$.

The goal of the present article is to introduce a similar construction of cohomology classes in which the matrix group $\SL_2(\Z[1/p])$ is replaced by certain arithmetic subgroups of indefinite quaternion algebras over totally real number fields. These cohomology classes can also be evaluated at appropriate algebraic elements in $\cH_p$. Inspired by the findings of \cite{darmon-vonk}, we also conjecture that these values are $p$-adic logarithms of algebraic numbers\footnote{The cohomology classes that we construct should be regarded as playing the role of the logarithmic derivatives of those in \cite{darmon-vonk}, hence the presence of the $p$-adic logarithm in our conjectures} lying in certain ring class fields, and we present experimental evidence in support of this expectation.

More precisely, in our construction the base field is allowed to be any totally real number field $F$ of narrow class number $1$. We fix a prime $\fp$ of $F$, and we consider quadratic extensions $K/F$ such that $\fp$ is inert in $K$ (so that, in particular, $K\setminus F$ is contained in the $\fp$-adic upper half plane $\cH_\fp$ associated to the completion of $F$ at $\fp$) and such that the set of real places of $F$ that split in $K$ consists of a single place $v_\infty$. These extensions are called Almost Totally Complex (ATC), because all but one real places of $F$ extend to complex places of $K$. The role played by the matrix algebra in \cite{darmon-vonk} is played here by a quaternion algebra $B/F$ which is split at $\fp$ and at $v_\infty$ (but it is allowed to ramify at other finite primes), and we consider the arithmetic group $\Gamma_0$ which is the group of norm $1$ units in a maximal order $R$ in $B$. 

In the setting of division quaternion algebras the sensible notion of ATC point in $\cH_\fp$ is defined by means of optimal embeddings, and we shall adopt this point of view (see Section \ref{sec:setup} below for precise definitions). If $K_1$ and $K_2$ are quadratic extensions of $F$ as above and $\cO_1\subset K_1$, $\cO_2\subset K_2$ are quadratic $\cO_F$-orders, we associate to any optimal embedding $\psi_1\colon \cO_1\hookrightarrow R$
cohomology classes   $\Phi_{\psi_1}^{\bullet}$ with $\bullet \in \{\mathrm{even}, \mathrm{odd}, +, -\}$
which can be evaluated, in an appropriate sense, at optimal embeddings $\psi_{2}\colon \cO_2\hookrightarrow R$. We conjecture that the resulting quantities $J_{\psi_1,\psi_2}^{\bullet}$
belong to the compositum of the ring class fields of $\cO_1$ and $\cO_2$, and we perform computer calculations of these quantities in concrete examples that lend credence to this expectation.

The motivation and interest for a construction following the lines of \cite{darmon-vonk} in a more general quaternionic setting is twofold. First of all, it allows to treat the case of quadratic extensions $K/F$ with base fields $F$ other than $\QQ$. For instance, in \S\ref{ex:real quadratic F} we provide numerical verifications of the algebraicity our constructions in an example  where $F$ is a real quadratic field and $K$ is a degree $4$ field of signature $(2,1)$. But there is an additional nice feature of having a construction that allows for different choices of the quaternion algebra even for a fixed extension $K/F$, and this can already be appreciated in the case $F=\QQ$. Indeed, in \S \ref{subsection:diferent ps} we present two local computations (using different primes $p$ and different quaternion algebras $B$) that compellingly give rise to the same global quantity $J_{\psi_1,\psi_2}^+$. This is in consonance with the prediction that these Darmon--Vonk-like cohomology classes, which are constructed purely as $p$-adic objects, might be just the local manifestations at different primes $p$ of an object of a global nature. The reader can consult \cite[\S 4]{dv2} for a more detailed explanation of this expectation, of which our example of \S\ref{subsection:diferent ps} can be regarded as the first numerical verification. 

Next we give an overview of the contents of the article and the structure of our construction which, even if it is inspired by that of \cite{darmon-vonk}, it follows a slightly different approach at some points. A first difference is that our construction only uses the arithmetic groups $\Gamma_0$ and $\Gamma_0(\fp)$ (the later one associated to an Eichler order of level $\fp$), rather than $p$-arithmetic groups, and the cohomology classes belong to $H^1(\Gamma_0(\fp),\Lambda)$ with $\Lambda\subset \bar\Z_\fp[\![ x]\!]$ a power series ring. In this aspect, our construction of the cohomology class attached to $\psi_1$ can be seen more akin to the computational strategy developed in \cite[\S2.5]{darmon-vonk} to numerically calculate the cohomology classes by means of the iteration of a $U_p$-operator. A second difference is in the evaluation of these cohomology classes. In \cite{darmon-vonk} the cohomology classes initially take values in $\cM^\times/\CC_p^\times$, and in order to be meaningfully evaluated at RM points in $\cH_p$ they need to be lifted to classes with values in $\cM^\times$. In our approach instead of lifting the cohomology class we lift the homology classes attached to $\psi_2$ to classes with values in divisors of degree $0$, at which the cocycles can be directly evaluated. These homology classes with values in $\Div^0\cH_\fp$ are in fact the same ones that are used in the construction of Stark--Heegner points in \cite{Gr} and \cite{gms1}.

In Section \ref{sec:setup} we describe in detail the setting for our constructions, we introduce some of the relevant ingredients and we fix some choices for them. In Section \ref{sec:homology} we recall how to associate an homology class in $H_1(\Gamma_0,\Div^0\cH_\fp)$ to any optimal embedding. In Section \ref{sec:cohomology} we associate to any optimal embedding a cohomology class in $ H^1(\Gamma_0,\Div^0\cH_\fp)$. This is the seed for the cohomology classes $\Phi_{\psi_1}^{\bullet}\in H^1(\Gamma_0(\fp),\Lambda)$ constructed in Section \ref{sec: overconvergent} by iterating the $U_\fp$ operator. In Section \ref{sec:pairing} we explain how to evaluate the cohomology classes at the homology classes (the evaluation pairing) and we formulate the conjecture that the values obtained in this way are logarithms of algebraic numbers. In Section \ref{sec:effective} we give an effective version of the construction of the cohomology classes that allow for their computer calculation. Finally, in Section \ref{sec:numerical} we provide numerical computations of our constructions in specific examples, and we verify that they convincingly satisfy the algebraicity conjecture.

Finally, it is worth pointing out that a related construction of quaternionic rigid meromorphic cocycles in a more abstract setting has been recently proposed by Gehrmann in \cite{lennart}.

\subsection*{Acknowledgments} We thank Jan Vonk and Henri Darmon for illuminating conversations about their work and for sharing with us the code used in their original calculations. We also thank them for suggesting to us the interest of the numerical verification that we present in  \S \ref{subsection:diferent ps}.  Guitart was partially funded by project PID2019-107297GB-I00. Xarles was partially supported by project MTM2016-75980-P. This work has received funding from the European Research Council (ERC) under the European Union's Horizon 2020 research and innovation programme (grant agreement No 682152).

\section{Setup}\label{sec:setup}

In this section we introduce the main objects that we will use in the construction of quaternionic $p$-adic singular moduli and some notation that will be in force throughout the article. Let $F$ be a totally real number field of narrow class number $1$, and let $B$ be an almost totally definite quaternion algebra over $F$. That is to say, the set of archimedean places of $F$ at which $B$ is split consists of a single place $v_\infty$.  Fix a finite prime $\fp$ of $F$ at which $B$ is split; let $\QQ_\fp$ denote the completion of $F$ at $\fp$ and write $\ZZ_\fp$ for the ring of integers of $\QQ_\fp$. Fix also embeddings
\[
\iota_\fp\colon B\injects M_2(\QQ_\fp) \ \mathrm{ and } \ \iota_\infty\colon B\injects M_2(\RR),
\]
induced by splittings of $B$ at $\fp$ and $v_\infty$ respectively. 
  
Let $p$ be the rational prime below $\fp$ and put $\CC_p=\widehat{\overline{\QQ}}_p$, the completion of an algebraic closure of $\QQ_p$. Denote by 
\[\cH=\PP^1(\CC)\setminus \PP^1(\RR)\] 
the complex upper half plane and by \[\cH_\fp=\PP^1(\CC_p)\setminus\PP^1(\QQ_\fp)\] 
the $\CC_\fp$-points of the $\fp$-adic upper half plane. We denote by $\Div \cH_\fp$ (resp. $\Div^0\cH_\fp$) the $\Z$-module of divisors (resp. degree $0$ divisors) on $\cH_\fp$. The algebra $B$ acts on $\cH$ and $\cH_\fp$ by fractional linear transformations via $\iota_\infty$ and $\iota_\fp$, respectively, and this induces actions on $\Div\cH_\fp$ and $\Div^0\cH_\fp$.

We choose a maximal order $R\subset B$ such that
  \[
  \iota_\fp (R\otimes \ZZ_\fp) = \mathrm{M}_2(\ZZ_\fp),
  \]   
  and let \[\Gamma_0=R^\times_1\subset B^\times\] be the group of norm one units of $R$.

Throughout the article,  $K$ will denote a quadratic extension of $F$ such that $\fp$ is inert in $K$, $v_\infty$ splits in $K$ and all other infinite places of $F$ are ramified. We say that $K$ is an almost totally complex quadratic extension of $F$, in the sense that all but one of the real places of $F$ extend to complex places of $K$. In particular, if $F=\QQ$ then $K$ is a real quadratic field.

Let $\cO_K$ be the ring of integers of $K$, and let $\cO\subset \cO_K$ be an $\cO_F$-order of $K$. An embedding of $F$-algebras $\psi\colon K\hookrightarrow B$ is said to be $(\cO,R)$-optimal if $\psi ^{-1}(R)= \cO$. We will denote by $\cE(\cO,R)$ the set of $(\cO,R)$-optimal embeddings.

 
\section{Darmon's homology classes}
\label{sec:homology}
 
 The aim of this section is to recall the construction of \cite[\S7]{Gr} and \cite[\S 4.1]{gms1} that associates to any optimal embedding  $\psi\in \cE(\cO,R)$ an homology class $c_\psi^0\in H_1(\Gamma_0, \Div^0 \cH_\fp)$. These homology classes are the same ones that appear in the construction of Darmon points on elliptic curves. Note that \cite{Gr} and \cite{gms1} work with more general quadratic extensions $K/F$ and construct classes that can lie in higher homology groups in general. For convenience of the reader, we next recall the construction of these classes in the setting of the present article (which corresponds to the case $n=0$ of \cite{Gr}).
 
 We denote by $\cO_1^\times$ the group of units of $\cO$ of relative norm one:
 \[
 \cO_1^\times = \{w\in \cO^\times \colon \mathrm{Norm}_{K/F}w=1\}.
 \] 
 By Dirichlet's unit theorem $\cO_1^\times$ is a group of rank $1$, and we fix $u$ to be a generator of the free part. Put $\gamma_\psi = \psi(u)$, which belongs to $\Gamma_0$. The matrix $\iota_\fp(\gamma_\psi)$ has two fixed points on $\cH_\fp$, that we denote by $\tau_{\psi,p}$, $\tau'_{\psi,p}$. In fact, they belong to $K_\fp\setminus \QQ_\fp$, and we label them in such a way that $\tau_{\psi,p}$ is the one that satisfies that 
 \[
 \iota_\fp(\psi(a))\left(\begin{array}{c} \tau_{\psi,p} \\ 1 \end{array}\right) = a \left(\begin{array}{c} \tau_{\psi,p} \\ 1 \end{array}\right)\ \text{ for all }a\in K,
 \]
 where the operation on the left hand side is matrix multiplication, and we are using a fixed embedding of $K$ into $K_\fp$ for the multiplication by scalars on the right.
 
 For concreteness and for simplifying the presentation of some constructions in further sections, we will make the assumption that $\tau_{\psi,p}$ belongs to the principal affinoid
 \[
 A_0 = \{z\in \CC_p\colon |z-t|\geq 1 \text{ for all }t\in \ZZ_\fp \text{ and } |z|\leq 1 \}.
 \]
 In fact, if we allow to replace the chosen maximal order in $B$ and the optimal embedding, we can always suppose that $\tau_{\psi,p}\in A_0$. Indeed, since $\fp$ does not ramify in $K$ there exists $g\in B^\times$ such that $g\tau_{\psi,p}\in A_0$, so the conjugated embedding \[g\psi 
 g^{-1}\colon B\hookrightarrow gRg^{-1}\] satisfies this property.
 
 The action of $\Gamma_0$ on $\cH_\fp$ preserves $A_0$, and we will consider the homology groups $H_1(\Gamma_0,\Div A_0)$ and $H_1(\Gamma_0,\Div^0A_0)$, which we regard as groups of $1$-cycles modulo $1$-boundaries. Recall that for $G$ a group and a $G$-module $V$ the groups of $1$-chains and $2$-chains are, respectively 
 \[
 C_1(G,V) = \Z[G]\otimes_\Z V\ \mathrm{ and }\  C_2(G,V) = \Z[G]\otimes_\Z\Z[G]\otimes_\Z V.
 \]
 The boundary maps 
 \[
 \partial_1\colon C_1(G,V)\lra V \ \mathrm{ and }\ \partial_2\colon C_2(G,V)\lra C_1(G,V)
 \]
 are given by
 \[
 \partial_1(g\otimes v)=gv - v\ \mathrm{ and } \ \partial_2(g\otimes h\otimes v) = h\otimes g^{-1}v-gh\otimes v +g\otimes v,
 \]
 and then $H_1(G,V)=\ker \partial_1/\operatorname{im} \partial_2$.
 
 We now consider the chain $\gamma_\psi\otimes\tau_{\psi,p}\in C_1(\Gamma_0,\Div A_0)$. Since $\gamma_\psi$ fixes $\tau_{\psi,p}$ its boundary vanishes:
 \[
 \partial(\gamma_\psi\otimes\tau_{\psi,p}) = \gamma_\psi\tau_{\psi,p}-\tau_{\psi,p} = 0.
 \]
 Thus it is a cycle, and we will denote by $c_\psi\in H^1(\Gamma_0,\Div A_0)$ the associated homology class.

The long exact sequence in homology associated to the degree map 
\begin{equation}\label{eq: degree map}
0\ra \Div^0A_0\lra \Div A_0\stackrel{\deg}{\lra} \ZZ\lra 0
\end{equation}
gives rise to 
\[
\cdots \lra H_2(\Gamma_0,\ZZ)\stackrel{\delta}{\lra} H_1(\Gamma_0,\Div^0 A)\lra H_1(\Gamma_0,\Div A_0)\stackrel{\deg}{\lra} H_1(\Gamma_0,\ZZ)\lra \cdots .  
\]
For any $B^\times$-module $V$, the homology groups $H_i(\Gamma_0,V)$ are equipped with the action of the Hecke operators $T_\fl$ and $S_\fl$, indexed by primes $\fl$ of $F$ not dividing the discriminant of $B$. This applies in particular to $H_1(\Gamma_0,\Div \cH_\fp)$ and $H_1(\Gamma_0,\ZZ)$.

The group $H_1(\Gamma_0,\Z)$ is torsion as a Hecke module, so there exists $T\in\mathbb{T}$ such that 
\[
T(\deg c_\psi)=0.
\]
One can always choose $T$ in such a way that it only involves Hecke operators away from $\fp$. These operators can be made to act on $H_1(\Gamma_0,\Div A_0)$, since elements of $R$ whose norm is a $\fp$-adic unit preserve $A_0$. This means that $Tc_\psi$ is a well defined element in $H_1(\Gamma_0,\Div A_0)$. Moreover, it satisfies that 
\[
\deg(Tc_\psi)=T\deg(c_\psi)=0.
\]
Thus $Tc_\psi$ can be lifted to an element $c_{\psi}^0 \in H_1(\Gamma_0,\Div^0 A_0)$. Observe that $c_\psi^0$ actually depends on the choice of $T$, but we will not make this dependence explicit in the notation. Also, even for a fixed choice of $T$ the class $c_\psi^0$ is only determined up to elements of $\delta(H_2(\Gamma_0,\Z))$. 

\begin{rmk}\label{rmk:no hecke}
When $H_1(\Gamma_0,\Z)$ is a torsion group (i.e., when the Shimura curve $X_B(1) = \Gamma_0\backslash \cH$ has genus $0$), we can take $T$ to be an integer. In other words, in this case there is no need to act by the Hecke algebra, it is enough to replace $c_\psi$ by an appropriate muliple in order to lift it to a cycle with values in $\Div^0A_0$.
\end{rmk}

\section{Cohomology classes for $\Gamma_0$}
\label{sec:cohomology}
The goal of this section is to attach to an optimal embedding $\psi\in\cE(\cO,R)$ a cohomology class 
\[
\varphi_\psi^0\in H^1(\Gamma_0,\Div^0\cH_\fp).
\]
Similarly as we did in \S\ref{sec:homology}, we will make the additional assumption that $\tau_{\psi,p}$ belongs to the principal affinoid $A_0$ and in this case the classes $\varphi_\psi^0$ actually belong to $H^1(\Gamma_0,\Div^0 A_0)$. Another similarity with the construction of the homology classes of the previous section is that we will start by constructing a cohomology class $\varphi_\psi\in H^1(\Gamma_0,\Div A_0)$, that is, with values in divisors rather than divisors of degree $0$, and then we will see that it can be lifted to $H^1(\Gamma_0,\Div^0 A_0)$ by applying suitable Hecke operators.

We regard cohomology groups as classes of inhomogeneous cocycles. That is, for a $G$-module $V$ the group of $1$-cochains is
\[
C^1(G,V) =\{\text{maps }\varphi\colon G\ra V \}.
\]
The subgroup of $1$-cocycles is
\[
Z^1(G,V) = \{\varphi \in C^1(G,V)\colon \varphi(gh) = \varphi(g) + g\varphi(h) \text{ for all }g,h\in G\},
\]
and the first cohomology group $H^1(G,V)$ is the quotient of $Z^1(G,V)$ by the subgroup generated by the 1-coboundaries, which are the maps of the form $\varphi(g)=gv - v$ for some $v\in V$.

In order to define $\varphi_\psi^0$ we will fix a base point. If $B$ is a division algebra this base point belongs to $\cH$, but if $B$ is the split matrix algebra over $\QQ$ it can also belong to $\PP^1(\QQ)$. In order to treat the two cases simultaneously, it is notationally convenient to define
\[
\cH^*=\begin{cases}
\cH\cup \PP^1(\QQ) \text{ if } B\cong M_2(\QQ)\\
\cH \text{ otherwise.}
\end{cases}
\]
Fix throughout a base point $x\in \cH^*$. For $a,b\in \cH\cup \PP^1(\RR)$ we denote by $C(a,b)$ the oriented geodesic segment from $a$ to $b$ in $\cH\cup \PP^1(\RR)$. 
 
The matrix $\iota_\infty(\gamma_\psi)$ acts on $\PP^1(\RR)$ with two fixed points. If $\iota_\infty(\gamma_\psi)=\smtx a b c d$ with $c\neq 0$ they are given by\footnote{we will assume that $c\neq 0$, as this can always be achieved by replacing $\psi$ by a conjugate if necessary} 
\[
\frac{a-d \pm \sqrt{(d-a)^2+4bc}}{2c}.
\]
We denote by $\tau_{\psi,\infty}$ (resp. $\tau'_{\psi,\infty}$) the one corresponding to the choice of the positive square root (resp. the negative square root).

In this section we will consider both the action of $\Gamma_0$ on $\PP^1(\RR)$ via $\iota_\infty$ and the action of $\Gamma_0$ on $\cH_\fp$ via $\iota_\fp$. Suppose that $w$ is an element of the orbit $\Gamma_0 \tau_{\psi,\infty}$. Then $w = \gamma_w \tau_{\psi,\infty}$ for some $\gamma_w\in \Gamma_0$ and we define $w_p=\gamma_w\tau_{\psi,p}\in \cH_\fp$, where $\tau_{\psi,p}$ is the element of $\cH_\fp$ defined in Section \ref{sec:homology}. Observe that $w_p$ does not depend on the choice of $\gamma_w$, for any other choice is of the form $\gamma_w\gamma$ for some $\gamma\in\langle  \gamma_\psi\rangle$ and $\gamma_\psi\tau_{\psi,p} =\tau_{\psi,p}$. We also define $w_p'=\gamma_w\tau'_{\psi,p}$. Similarly, we define $w' = \gamma_w \tau'_{\psi,\infty}\in\PP^1(\RR)$.

 For $\gamma\in\Gamma_0$ and $w\in \Gamma_0\tau_{\psi,\infty}$ we define $\delta_\gamma(w)\in \{-1,0,1\}$ to be the signed intersection number between the oriented geodesics $C(x,\gamma x)$ and  $C(w,w')$. In particular,
 \[
\delta_\gamma(w)=\begin{cases}
\pm 1&\text{ if }C(x,\gamma x)\cap C(w,w')\neq \emptyset,\\
0&\mathrm{else.}
\end{cases}
\] 
The following is a key result for the construction.

\begin{lemma}\label{lemma:finiteness of deltas}
For each $\gamma\in\Gamma_0$, the set
\[
\{w\in \Gamma_0\tau_{\psi,\infty} ~|~ \delta_\gamma(w)\neq 0\}
\]
is finite.
\end{lemma}
\begin{proof}

 Let $w\in\Gamma_0\tau_{\psi,\infty}$ such that $\delta_\gamma(w)\neq 0$. Write $\gamma_w$ for an element in $\Gamma_0$ such that $w=\gamma_w\tau_{\psi,\infty}$. Then
\[
C( \gamma_w\tau_{\psi,\infty}, \gamma_w\tau_{\psi,\infty}')\cap C(x,\gamma x)\neq\emptyset
\]
and therefore
\[
C(\tau_{\psi,\infty},\tau'_{\psi,\infty})\cap  \gamma_w^{-1} C(x,\gamma x)\neq\emptyset
\]
Let $S$ be a fundamental domain for the action of $\langle \gamma_\psi\rangle $ on the geodesic $C(\tau_{\psi,\infty},\tau'_{\psi,\infty})$. Then there exists an integer $n$ such that
\[
\gamma_\psi^n S\cap  \gamma_w^{-1} C(x,\gamma x)\neq\emptyset.
\]
Therefore, in order to prove the statement of the lemma it is enough to show that there are finitely many $n\in\ZZ$ and $w\in\Gamma_0\tau_{\psi,\infty}$ such that 
\[
 S\cap \gamma_\psi^{-n} \gamma_w^{-1} C(x,\gamma x)\neq\emptyset.
\]
This follows from the fact that the action of $\Gamma_0$ on $\cH$ is discrete and proper: for compact sets $K$ and $L$, the set
\[
\{\gamma\in\Gamma_0 ~|~ \gamma K \cap L \neq\emptyset\}
\]
is finite (see~\cite[Theorem 34.5.1.]{voight-quaternion}), if $x\in\cH$, since then $C(x,\gamma x)$ is compact. The case that $x\in \cH^*\setminus \cH=\PP^1(\QQ)$, so $B \cong \mathrm{M}_2(\QQ)$, was proved in \cite[\S 1.4]{darmon-vonk} by a different argument.
\end{proof}
 
For $\gamma\in \Gamma_0$ define
\begin{equation}\label{eq: varphi_tau}
\varphi_\psi(\gamma) = \sum_{w\in\Gamma_0\tau_{\psi,\infty}} \delta_\gamma(w)\cdot w_p\in \Div A_0.
\end{equation}
The fact that $\varphi_\psi(\gamma)$ belongs to $\Div A_0$, i.e., that the above is a finite sum, is granted by Lemma \ref{lemma:finiteness of deltas}. Next, we will show that $\varphi_\psi$ is a one-cocycle of $\Gamma_0$ with values on $ \Div A_0$. The cocycle relation is
\[
\sum_{w\in\Gamma_0\tau_{\psi,\infty}} \delta_{\gamma_1\gamma_2}(w)\cdot w_p = \sum_{w\in\Gamma_0\tau_{\psi,\infty}} \delta_{\gamma_1}(w)\cdot w_p + \sum_{w\in\Gamma_0\tau_{\psi,\infty}} \delta_{\gamma_2}(w)\cdot \gamma_1 w_p,
\]
which relabeling the sum in the third term can be written as
\[
\sum_{w\in\Gamma_0\tau_{\psi,\infty}} \delta_{\gamma_1\gamma_2}(w)\cdot w_p = \sum_{w\in\Gamma_0\tau_{\psi,\infty}} \delta_{\gamma_1}(w)\cdot w_p + \sum_{w\in\Gamma_0\tau_{\psi,\infty}} \delta_{\gamma_2}(\gamma_1^{-1}w)\cdot w_p.
\]
Thus we are reduced to prove the following lemma.
\begin{lemma}\label{lemma:cocycle condition}
For $\gamma_1,\gamma_2\in\Gamma_0$ and $w\in\Gamma_0\tau_{\psi,\infty}$ we have that
\begin{align}\label{eq: cocycle relation}
\delta_{\gamma_1\gamma_2}(w) = \delta_{\gamma_1}(w) + \delta_{\gamma_2}(\gamma_1^{-1}w).
\end{align}
\end{lemma}
\begin{proof}
The term $\delta_{\gamma_2}(\gamma_1^{-1}w)$ is given by the intersection number between the geodesics
\[
C(\gamma_1^{-1}w,\gamma_1^{-1} w') \text{ and }C(x,\gamma_2 x).
\]
This clearly coincides with the intersection number between $C(w,w')$ and $C(\gamma_1 x,\gamma_1\gamma_2 x)$. Now the identity \eqref{eq: cocycle relation} is just a reflection of the fact that the geodesic $C(w,w')$ intersects either $0$ or $2$ sides of the triangle with vertices $x$, $\gamma_1 x$ and $\gamma_1\gamma_2 x$. 
\end{proof}

The next lemma shows that the cohomology class attached to the cocycle $\varphi_\psi$ does not depend on the choice of the auxiliary point $x$. To state the result, let us temporarily denote by $\varphi_{\psi,x}$ the map defined as in \eqref{eq: varphi_tau} corresponding to the choice of $x$ as auxiliary point. Also, for $x,y\in\cH^*$ and $w\in \Gamma_0\tau_{\psi,\infty}$ denote by $\varepsilon_{x,y}(w)$ the signed intersection number between $C(x,y)$ and $C(w,w')$, and define $D_{x,y}\in \Div^0(A_0)$ as 
\[
D_{x,y} = \sum_{w\in\Gamma_0\tau_{\psi,\infty}} \eps_{x,y}(w) \cdot w_p.
\]
\begin{lemma}
We have that
\begin{equation}\label{eq: coboundary}
\varphi_{\tau,y}(\gamma) - \varphi_{\tau,x}(\gamma) = \gamma D_{x,y} - D_{x,y}.
\end{equation}
\end{lemma}
\begin{proof} 
In order to prove \eqref{eq: coboundary}, which is an identity of divisors, it is enough to check that for any $w\in \Gamma_0\tau_{\psi,\infty}$ the coefficient accompanying $w$ in both sides of the equality is the same. Therefore, we need to check that
\begin{equation}\label{eq: check the coboundary}
\eps_{y,\gamma y}(w) - \eps_{ x,\gamma x}(w) = \eps_{x,y}(\gamma^{-1}w)  - \eps_{ x, y}(w)  \text{ for all } w\in \Gamma_0\tau_{\psi,\infty}.
\end{equation}

Consider the edges $C(x,\gamma x)$, $C(\gamma x, \gamma y)$, $C(\gamma y, y)$  and $C(y, x)$. Since these four edges form a closed curve in $\cH^*$, the signed intersection number of such curve with $C(w,w')$ is $0$. This translates into the relation
\[
\eps_{x,\gamma x}(w) + \eps_{\gamma x,\gamma y}(w) + \eps_{\gamma y,y}(w) + \eps_{y, x}(w) = 0 \text{ for all } w\in \Gamma_0\tau_{\psi,\infty}.
\]
From this we deduce that
\[
\eps_{y,\gamma y}(w) - \eps_{ x,\gamma x}(w) = \eps_{\gamma x,\gamma y}(w)  - \eps_{ x, y}(w)  \text{ for all } w\in \Gamma_0\tau_{\psi,\infty},
\]
and since $\eps_{\gamma x,\gamma y}(w) = \eps_{x,y}(\gamma ^{-1}w)$ we obain \eqref{eq: check the coboundary} as we aimed.
\end{proof}

We have thus proved the following result.
\begin{prop}
The assignment $\gamma\mapsto \varphi_\psi(\gamma)$ defines a cocycle in $Z^1(\Gamma_0,\Div A_0)$. Moreover, the corresponding cohomology class (which we call $\varphi_\psi$ as well) is independent of the choice of base point $x$ that was made to define it.
\end{prop}
Similarly as in Section \ref{sec:homology} we would like to see that $\varphi_\psi$ lies in the image of the natural map 
\begin{equation}\label{eq: image of the natural map}
H^1(\Gamma_0,\Div^0 A_0) \lra H^1(\Gamma_0,\Div A_0).
\end{equation}
That is to say, we would like to see that $\varphi_\psi$ can be lifted to a cocycle with values in divisors of degree $0$. This might not be true in general, and we thus need to act by an appropriate Hecke operator. More precisely, the long exact sequence in cohomology associated to \eqref{eq: degree map} is
\[
\cdots \lra H^1(\Gamma_0,\Div^0 H_\fp) \lra H^1(\Gamma_0,\Div H_\fp)\stackrel{\deg}{\lra} H^1(\Gamma_0,\ZZ)\lra\cdots .
\]
There exists a Hecke operator $T$ such that $T\deg(\varphi_\psi)=0$. Therefore, $T\varphi_\psi$ can be lifted to a cohomology class with values in divisors of degree $0$. Moreover, $T$ can be chosen in such a way that it only involves Hecke operators $T_\fl$ and $S_\fl$ for $\fl\neq \fp$, and then $T\varphi_\psi$ lifts to a cohomology class $\varphi_\psi^0$ belonging to $H^1(\Gamma_0,\Div^0 A_0)$. Once again, we do not make explicit the dependence of $\varphi_\psi^0$ on $T$ in the notation.

\begin{rmk}\label{rmk:no hecke 2}
Here a similar comment as in Remark \ref{rmk:no hecke} applies. Whenever $\deg(\varphi_\psi)$ is torsion, we will take $T$ to be an integer. This will always be the case when $H^1(\Gamma_0,\Z)$ is torsion (i.e., when the Shimura curve $\Gamma_0\backslash \cH$ has genus $0$). 
\end{rmk}

\section{Overconvergent cohomology classes}\label{sec: overconvergent}

Let $R(\fp)\subset R$ be an Eichler order of level $\fp$ such that
\[
\iota_\fp(R(\fp)\otimes \QQ_\fp) = \{\smtx a b c d \in \mathrm{M}_2(\ZZ_\fp)\colon c\in\fp\},
\]
 and denote by $\Gamma_0(\fp)$ the group of norm $1$ units of $R(\fp)$. Let $\varpi$ be a totally positive generator of $\fp$, and let  $\Lambda = \bar\Z_\fp[\![\varpi x]\!]$. The goal of this section is to associate to the cohomology class $\varphi_\psi^0$ constructed in Section \ref{sec:cohomology} four cohomology classes \[
 \Phi_\psi^\mathrm{even},\Phi_\psi^\mathrm{odd},\Phi_\psi^+,\Phi_\psi^-\in H^1(\Gamma_0(\fp),\Lambda).
 \]
 These are the cohomology classes that will be paired with the homology classes $c_\psi^0$ of Section \ref{sec:homology}. Our definition of these cohomology classes and the pairing of Section \ref{sec:pairing} are inspired by the construction of \cite[\S 3.5]{darmon-vonk}, where the authors obtain a formula for their rigid meromorphic cocycles  which is suitable for efficient calculations in terms of the iteration of the $U_p$-operator.
  
 Let $\cT$ be the Bruhat--Tits tree of $\GL_2(\QQ_\fp)$, and denote by $\cV$ its set of vertices and by  $\cE$ its set of edges. It is a $(|\fp| + 1)$-regular tree (here $|\fp|$ denotes the cardinality of $\FF_\fp=\Z_\fp/\fp\Z_\fp$), whose vertices are in bijection with similarity classes of $\ZZ_\fp$-lattices in $\QQ_\fp^2$ and two vertices are connected by an edge if they admit representing lattices $\Lambda_1,\Lambda_2$ with $\fp \Lambda_1\subsetneq \Lambda_2\subsetneq \Lambda_1$. Let $v_0\in \cV(\cT_\fp)$ be the vertex corresponding to the lattice $\ZZ_\fp^2$.  There is a $\GL_2(\QQ_\fp)$-equivariant reduction map
\[
\mathrm{red}\colon \cH_\fp \lra \cT=\cV\cup\cE
\]
such that $\mathrm{red}^{-1}(v_0)=A_0$. 
Denote by $\cV_1=\{v_t\}_{t\in \PP^1(\FF_\fp)}$ the set of the $|\fp| + 1$ vertices at distance $1$ of $v_0$. Removing from $\cT$ the vertex $v_0$ and the (open) edges originating from $v_0$ one obtains $|\fp|+1$ subtrees $\{\cT_t\}_{t\in\PP^1(\FF_\fp)}$ originating at each one of the vertices in $\cV_1$. Denote by 
\[
U_\infty = \{ z\in \cH_\fp\colon |z|\geq |\fp| \}
\]
and for $t$ in a set of representatives of $\ZZ_\fp/\fp\ZZ_\fp$ put
\[
U_t = \{ z\in \cH_\fp\colon |z-t|\leq |\fp| \}.
\]
We label the vertices of $\cV_1$ in such a way that $\red^{-1}(v_t)\subset U_t$.

The group $\Gamma_0$ acts on $\red^{-1}(\{\cT_t\})$ permuting the sets $U_t$, and this induces an isomorphism of $\Gamma_0$-modules
\begin{equation}\label{eq: induced}
\bigoplus_{t\in\PP^1(\FF_\fp)} \Div^0(U_t)\simeq \Ind_{\Gamma_0(\fp)}^{\Gamma_0}\Div^0 U_\infty.
\end{equation}

We next recall the definition of the Hecke operator $W_\fp$ acting on $H^1(\Gamma_0,\Div \cH_\fp)$. Let $\omega_\fp\in R(\fp)$ be and element that normalizes $\Gamma_0(\fp)$ and which is locally of the form 
\[
\iota_\fp(\omega_\fp) = u'\smtx{0}{-1}{\varpi}{0},
\]
for some $u'\in \SL_2(\ZZ_\fp)$ whose lower left entry belongs to $\fp$. If $\varphi$ is a cocycle representing an element in $H^1(\Gamma_0(\fp),\Div\cH_\fp)$ then
\begin{equation}\label{eq: W_fp}
(W_\fp \varphi) (g) = \omega_\fp \varphi(\omega_\fp^{-1}g\omega_\fp).
\end{equation}

Recall the cohomology class $\varphi_\psi^0\in H^1(\Gamma_0,\Div^0(A_0))$ constructed in \S\ref{sec:cohomology} and define $\tilde\phi_\psi = W_\fp\varphi_\psi^0$. Since $\varphi_\psi^0$ takes values in $\Div^0(A_0)$, from the explicit formula for $W_\fp$ in \eqref{eq: W_fp} we see that $\tilde\phi_\psi$ takes values in $\Div^0(U_\infty)$.

Let $\Sigma_0(\fp)\subset R(\fp)$ be the multiplicative semigroup formed by elements $\gamma$ of non-zero norm and such that $\iota_\fp(\gamma)=\smtx a b c d$ with $c\in\fp$ and $a\in\ZZ_\fp^\times$. Let $\bar\Z_\fp\langle x \rangle$ be the Tate algebra over $\bar\ZZ_\fp$, which we endow with a weight two action by $\Sigma_0(\fp)$ by means of the following formula:
\begin{equation}\label{eq: action}
\sigma f(x) = (ad-bc)(a - cx)^{-2} f(\sigma^{-1}x),\quad \text{ for }f(x)\in \bar\Z_\fp\langle x \rangle \text{ and } \sigma=\smtx abcd\in\Sigma_0(\fp).
\end{equation}
Consider also the $\Sigma_0(\fp)$-stable submodule $\Lambda\subseteq  \bar\Z_\fp\langle x \rangle$, consisting of functions converging on the open ball of radius $|\fp|$. Note that $\Lambda$ is identified with the subring $\bar\Z_\fp[\![\varpi x]\!]\subset \bar\Z_\fp\langle x\rangle$.

There is a natural $\Sigma_0(\fp)$-equivariant map $\eta\colon \Div^0 U_\infty \to \Lambda$, induced by
\[
\eta (Q-P) =  \operatorname{dlog}\left(\frac{x-Q}{x-P}\right)=\frac{1}{x-Q} - \frac{1}{x-P}.
\]
Equivariance of the above map is straightforward to check directly. Alternatively, one can notice that for $g\in\Sigma_0(\fp)$ the functions $\left(\frac{x-gQ}{x-gP}\right)$ and $\left(\frac{g^{-1}x-Q}{g^{-1}x-P}\right)$ have the same zeros and poles so they differ by a multiplicative constant; their logarithms hence differ by an additive constant. Taking the $\dlog$ kills this constant and changes the weight by two.

We see that $\eta$ defines a map in cohomology, which we will denote by the same letter
\[
\eta\colon H^1(\Gamma_0(\fp),\Div^0 U_\infty)\lra H^1(\Gamma_0(\fp), \Lambda),
\]
and we define $\phi_\psi\in H^1(\Gamma_0(\fp), \Lambda)$ to be $\phi_\psi = \eta(\res \tilde\phi_\psi)$, where 
\[\res\colon H^1(\Gamma_0,\Div^0 U_0)\to H^1(\Gamma_0(\fp),\Div^0 U_0)\] 
is the natural restriction map.

We consider now the operator $U_\fp$ acting on $H^1(\Gamma_0(\fp),\Lambda)$.  Let $h_\varpi\in\Sigma_0(\fp)$ ben an element of reduced norm $\varpi$ such that $\iota_\fp(h_\varpi) = \smtx 1 0 0 \varpi u_\varpi$ for some $u_\varpi \in \SL_2(\Z_\fp)$. There is a double coset decomposition
\[
\Gamma_0(\fp) h_\varpi \Gamma_0(\fp) = \bigsqcup_{t\in \FF_\fp}h_t\Gamma_0(\fp), 
\]
where the elements $h_t$ can be chosen in such a way that
\[
  \iota_\fp(h_t) = \smtx{1}{t}{0}{\varpi}u_t,\text{ for $t$ in a system of representatives of }\FF_\fp=\ZZ_\fp/\fp\ZZ_\fp,
\]
and where the $u_t$ belong to $\SL_2(\ZZ_\fp)$ and have lower left entry belonging to $\fp$. Then for $t\in \FF_\fp$ and $\gamma\in \Gamma_0(\fp)$, there is a unique $b$ and a unique $s_t(\gamma)\in\Gamma_0(\fp)$ such that
\[
h_t s_t(\gamma) = \gamma  h_{b}.
\]
If $\varphi$ is a cocycle representing an element in $H^1(\Gamma_0,\Lambda)$ then $U_\fp\varphi$ is represented by the cocycle given by
\begin{equation}\label{eq: U_fp}
(U_\fp\varphi)(\gamma) = \sum_{t\in \FF_\fp } h_t \varphi(s_t(\gamma)).
\end{equation}
\begin{lemma}\label{lemma: up converges}
For any $\gamma\in\Gamma_0(\fp)$ we have that $(U_\fp\varphi)(\gamma)$ belongs to $\varpi\Lambda$.
\end{lemma}
\begin{proof}
The fact that $\varphi(\gamma)$ belongs to $\Lambda$ means that we can write $\varphi(\gamma)=\sum_{n\geq 0}a_n\varpi^nx^n$ for certain coefficients $a_n\in\bar\ZZ_\fp$. Then the result follows directly from \eqref{eq: U_fp} and the action \eqref{eq: action} of $h_t$ on elements of $\Lambda$. Indeed, $\iota_\fp(h_t)=\smtx{a}{b}{c}{d}$ with $c\in\fp$, and therefore $h_t^{-1}x$ and $(a-cx)^{-2}$ belong to $\Lambda$. This implies that 
\[
(a-cx)^{-2}\varphi(s_t(\gamma))(h_t^{-1}x)
\]
belongs to $\Lambda$, and the fact that $\det(h_t)=\varpi$ implies that $h_t\varphi(s_t(\gamma))$ belongs to $\varpi\Lambda$.
\end{proof}

Finally, we can define the cohomology classes $\Phi_\psi^{\bullet}$. We will define first $\Phi_\psi^+$ and $\Phi_\psi^-$, which are given as explicit series involving the $U_\fp$-iterates of $\phi_\psi$.
\begin{prop}
The series given by
\[
\Phi_\psi^+= \phi_\psi + U_\fp \phi_\psi + U_{\fp}^2 \phi_\psi+ U_{\fp}^3 \phi_\psi + U_{\fp}^4 \phi_\psi  + \cdots,
\]
\[
\Phi_\psi^-= -\phi_\psi + U_\fp \phi_\psi - U_{\fp}^2 \phi_\psi+ U_{\fp}^3 \phi_\psi - U_{\fp}^4 \phi_\psi  + \cdots 
\]
give rise to well defined cohomology classes in  $H^1(\Gamma_0(\fp),\Lambda)$. In addition, they satisfy that \[U_\fp\Phi_\psi^+ +  \phi_\psi =\Phi_\psi^+\ \mathrm{ and } \ U_\fp\Phi_\psi^- + \phi_\psi =-\Phi_\psi^-.
\]
\end{prop}
\begin{proof}
By Lemma \ref{lemma: up converges} for any $\gamma\in\Gamma_0(\fp)$ the series in $\Lambda$

\[
 \phi_\psi(\gamma) + (U_\fp \phi_\psi)(\gamma) + (U_{\fp}^2 \phi_\psi)(\gamma) + (U_{\fp}^3 \phi_\psi)(\gamma) + (U_{\fp}^4 \phi_\psi) (\gamma) + \cdots
\]
converges to an element of $\Lambda$. Therefore, the series
\[
\Phi_\psi^+= \phi_\psi + U_\fp \phi_\psi + U_{\fp}^2 \phi_\psi+ U_{\fp}^3 \phi_\psi + U_{\fp}^4 \phi_\psi  + \cdots
\]
gives rise to a well defined cohomology class $\Phi_\psi^+\in H^1(\Gamma_0(\fp),\Lambda)$, which clearly satisfies that $U_\fp\Phi_\psi^+ +  \phi_\psi =\Phi_\psi^+$. The same argument applies to $\Phi_\psi^-$. 

\end{proof}

We also set
\begin{align*}
\Phi_\psi^{\mathrm{even}} &= \frac 12 \left(\Phi_\psi^+ - \Phi_\psi^-\right) = \phi_\psi + U_{\fp}^2 \phi_\psi + U_{\fp}^4 \phi_\psi  + \cdots,\quad\text{and}\\
\Phi_\psi^{\mathrm{odd}} &=  \frac 12 \left(\Phi_\psi^+ + \Phi_\psi^-\right) =  U_{\fp} \phi_\psi + U_{\fp}^3 \phi_\psi  + \cdots.
\end{align*}

\section{Evaluation pairing and algebraicity conjecture}\label{sec:pairing}
There is a natural $\Gamma_0(\fp)$-equivariant integration pairing
\[
\begin{array}{ccc} \Lambda \times \Div^0 A_0 & \lra & \CC_\fp\\  (f, Q-P) & \longmapsto  & \int_{P}^{Q} f(x)dx,\end{array}
\]
where we define $\int_{P}^{Q} f(x)dx = F(P)-F(Q)$ for any primitive $F$ of $f$. Observe that this is well defined for $f\in\Lambda$ and $P,Q\in A_0$ since $f$ belongs to $\Lambda$ and therefore any primitive $F$ belongs to $\bar\ZZ_\fp\langle x\rangle$. We will denote by $\langle\cdot,\cdot\rangle$ the induced integration pairing in (co)homology
\[
\langle\cdot,\cdot\rangle\colon H^1(\Gamma_0(\fp),\Lambda)\times H_1(\Gamma_0(\fp),\Div^0 A_0)\lra \CC_\fp.
\]
Denote by $\log_p\colon \CC_p^\times\lra \CC_p$ a choice of $p$-adic logarithm. Consider also the pairing
\[
\begin{array}{ccc}\CC_p(x)^\times/\CC_p^\times \times \Div^0 A_0 & \lra & \CC_\fp\\ (\bar f, Q-P)&\longmapsto & \log_p\left(f(Q)/f(P)\right),\end{array}
\]
where $f$ is any representative of $\bar f$ (observe that the result does not depend on the choice of representative).  Denote by $[\cdot,\cdot]$ the induced pairing in $\Gamma_0$-(co)homology
\[
[ \cdot,\cdot ]\colon H^1(\Gamma_0,\CC_p(x)^\times/\CC_p^\times)\times H_1(\Gamma_0,\Div^0 A_0)\lra \CC_p.
\]
Let now $K_1$ and $K_2$ be quadratic extensions of $F$, satisfying the same hypotheses that we have been assuming for the field denoted by $K$ so far. That is, for $i=1,2$, $K_i$ is a quadratic extension of $F$ such that $\fp$ is inert in $K_i$, the place $v_\infty$ is split in $K$ and all other infinite places of $F$ are ramified in $K$. For $i=1,2$, let $\cO_i$ be an $\cO_F$-order of $K_i$ and let $\psi_i\in\cE(\cO_i,R)$ be optimal embeddings. We can consider the homology class $c_{\psi_2}^0\in H_1(\Gamma_0,\Div^0 A_0)$ defined in Section~\ref{sec:homology}, and the cohomology classes $\varphi_{\psi_1}^0$ and $\Phi^\pm_{\psi_1}$, $\Phi_{\psi_1}^{\mathrm{even,odd}}$ defined in Sections~\ref{sec:cohomology} and \ref{sec: overconvergent}. Consider also the corestriction map in group homology
\[
\operatorname{cores}\colon H_1(\Gamma_0,\Div^0 A_0)\to H_1(\Gamma_0(\fp),\Div^0 A_0).
\]
Define the quantities
\[
J^\bullet_{\psi_1, \psi_2} = [ \varphi_{\psi_1}^0,c^0_{\psi_2}]+ \langle \Phi^{\bullet}_{\psi_1},\operatorname{cores} c^0_{\psi_2} \rangle,\quad\bullet\in \{\mathrm{even}, \mathrm{odd}, +, -\}.
\]
All these quantities belong to $\CC_p$, but we expect them to be actually $p$-adic logarithms of algebraic numbers. We will make the conjecture explicit only in the case where the Hecke operators used in the definition of $c_{\psi_2}^0$ and $\varphi_{\psi_1}^0$ are just integers as in  Remark \ref{rmk:no hecke} and Remark \ref{rmk:no hecke 2} (and in particular the classes $\deg(c_{\psi_2})$ and $\deg(\varphi_{\psi_1})$ are torsion). Recall that this can always be achieved if $H_1(\Gamma_0,\Z)$ is a torsion group (i.e., if the Shimura curve associated to $\Gamma_0$ has genus $0$).

\begin{conjecture}
\label{conj:main}
Suppose that the Hecke operators used in the definition of $\varphi_{\psi_1}^0,c_{\psi_2}^0$ are integers. Let $H_i$ be the narrow ring class field of $\cO_i$, and let $H = H_1 H_2$. There exist elements $P_{\psi_1,\psi_2}^\mathrm{even},P_{\psi_1,\psi_2}^\mathrm{odd},P_{\psi_1,\psi_2}^+,P_{\psi_1,\psi_2}^-\in H$ such that 
\[
J^\mathrm{even}_{\psi_1, \psi_2} =  \log_p(P_{\psi_1,\psi_2}^\mathrm{even}),\ \ J^\mathrm{odd}_{\psi_1, \psi_2} =  \log_p(P_{\psi_1,\psi_2}^\mathrm{odd}),\ \
J^+_{\psi_1, \psi_2}  = \log_p(P_{\psi_1,\psi_2}^+),\ \
J^-_{\psi_1, \psi_2}  = \log_p(P_{\psi_1,\psi_2}^-).
\]
\end{conjecture}

\begin{rmk}
The first term $[ \varphi_{\psi_1}^0,c^0_{\psi_2}]$ is obviously the logarithm of an algebraic number, by definition, and therefore one could formulate an equivalent conjecture with the quantity $\langle \Phi^\mathrm{even}_{\psi_1},\operatorname{cores} c^0_{\psi_2}\rangle$ instead of $J_{\psi_1,\psi_2}^\mathrm{even}$ (and similarly for the other classes). However, in the numerical experiments reported in Section \ref{sec:numerical} below one observes that the quantities $J_{\psi_1,\psi_2}^\mathrm{even,odd},J_{\psi_1,\psi_2}^\pm$ tend to give rise to algebraic numbers of smaller height, which makes them easier to recognize. This explains the presence of the term $[ \varphi_{\psi_1}^0,c^0_{\psi_2}]$ in the definition of these quantities.
\end{rmk}


\section{Effective computation}\label{sec:effective}
In this section we describe how one can effectively compute $J_{\psi_1,\psi_2}^\pm$. In fact, the key point is the calculation of $\varphi_\psi$, since the rest of the construction is either completely explicit and straightforward to implement or, as in the case of lifting (co)homology classes with values in $\Div A_0$ to classes with values in $\Div^0A_0$, has been worked out in \cite[\S4]{shpquat}.

In order to effectively compute $\varphi_\psi$ we need an explicit proof of Lemma~\ref{lemma:finiteness of deltas}. So suppose given $\gamma\in \Gamma_0$. We will describe how to compute a finite set $M_\gamma\subset \Gamma_0\tau_{\psi,\infty}$ with the property that
\[
\delta_\gamma(w)\neq 0\implies w\in M_\gamma.
\]
That is, the support of $\varphi_\psi(\gamma)$ is contained in $M_\gamma$. A finite computation of the values $\delta_\gamma(w)$ for all $w\in M_\gamma$ will determine $\varphi_\psi(\gamma)$.

Recall that the cocycle $\varphi_\psi$ depends on the choice of an auxiliary point $x\in \cH^*$. With the same notation as in Section~\ref{sec:cohomology}, let $S$ be a fundamental domain for the action of $\langle\gamma_\psi\rangle$ on the geodesic $C(\tau_{\psi,\infty},\tau'_{\psi,\infty})$. It is the form $C(y,\gamma_\psi y)$, where $y\in \cH$ is any choice of a point on $C(\tau_{\psi,\infty},\tau'_{\psi,\infty})$. Using standard algorithms coming from the computation with hyperbolic domains for fuchsian groups, we can find finite sets of elements in $\Gamma_0$, say $\{g_i\}_{i\in I}$ and $\{h_j\}_{j\in J}$ such that
\[
C(x,\gamma x)\subset \bigcup_{i\in I} g_i \bar\cD,\quad S\subset \bigcup_{j\in J} h_j \bar\cD,
\]
where $\cD$ is an open fundamental domain for the action of $\Gamma_0$ on $\cH$ (in particular for all $\gamma\in\Gamma_0$ we have $\gamma\cD\cap\cD\neq\emptyset\implies \gamma=1$). Then:
\begin{align*}
w\in \operatorname{Supp}(\varphi_\psi(\gamma)) &\iff C(x,\gamma x)\cap C(\gamma_w^{-1}\tau_{\psi,\infty},\gamma_w^{-1}\tau'_{\psi,\infty})\neq \emptyset\\
&\iff \exists n: C(\gamma_wx,\gamma_w\gamma x)\cap \gamma_\psi^{-n} S \neq \emptyset\\
&\iff \exists n : C(\gamma_\psi^{n}\gamma_wx,\gamma_\psi^{n}\gamma_w\gamma x)\cap S \neq \emptyset.
\end{align*}
Since the choice of $\gamma_w$ is only well-defined up to $\langle\gamma_\psi \rangle$, we may replace $\gamma_w$ with $\gamma_\tau^n\gamma_w$ to obtain
\[
w\in\operatorname{Supp}(\varphi_\psi(\gamma))\iff w = \gamma_w\tau_{\psi,\infty}\mathrm{, with } \gamma_w C(x,\gamma x)\cap S \neq\emptyset.
\]
Next, note that
\[
\gamma_w C(x,\gamma x)\cap S \subset \bigcup_{i\in I} \gamma_w g_i\overline\cD\cap \bigcup_{j\in J} h_i\overline\cD = \bigcup_{i,j} \gamma_wg_i\overline\cD \cap h_i\overline\cD.
\]
From the fact that $\cD$ is a fundamental domain, we deduce that each of the terms in the above right-hand side is empty unless $\gamma_wg_i = h_j$, that is unless $\gamma_w= \sigma h_jg_i^{-1}$, with $\sigma$ belonging to the finite set $\Sigma$ of elements giving the side pairing:
\[
\Sigma=\{\sigma\in\Gamma_0~|~\sigma\overline\cD\cap\overline\cD\neq\emptyset\}.
\]

To sum up, we have proved:
\begin{prop}
With the notation defined above,
\[
\operatorname{Supp}(\varphi_\psi(\gamma)) \subseteq \{ g_ih_j^{-1}\sigma^{-1}\tau \}_{i\in I, j\in J,\sigma\in\Sigma.}
\]
\end{prop}

\section{Numerical Evidence}\label{sec:numerical}
We collect below a sampling of various examples that give evidence of the conjecture. The aim is not to be exhaustive, and the interested reader is encouraged to try to find more examples on their own. The implementation\footnote{Available at \url{github.com/mmasdeu/darmonvonk}} is written in \texttt{Sage} (\cite{sage}) and heavily depends on the \texttt{darmonpoints} package\footnote{Maintained by the second named author at \url{github.com/mmasdeu/darmonpoints}}.

\subsection{A detailed quaterionic example over $\QQ$}

Consider the quaternion algebra $B/\QQ$ of discriminant $6$ given by $B=\left(\frac{6,-1}{\QQ}\right)$. In this example we take $p = 5$. We consider the maximal order $R = \langle 1, i, j, \frac{1+i+j+k}{2}  \rangle$. We consider also the quadratic field $K_1=\QQ(\sqrt{53})$, and the embedding
\[
\psi_1\colon \cO_{K_1}\injects R,\quad \frac{1+\sqrt{53}}{2}\mapsto  1/2 - 3/2i - 1/2j.
\]
This yields the hyperbolic element $\gamma_{\psi_1}$:
\[
\gamma_{\psi_1} = 51/2 + 21/2i + 7/2j.
\]
Similarly, we consider the quadratic extension $K_2=\QQ(\sqrt{23})$, and the embedding
\[
\psi_2\colon \cO_{K_2}\injects R,\quad \sqrt{23}\mapsto  2i + j.
\]
This yields the hyperbolic element $\gamma_{\psi_2}$:
\[
\gamma_{\psi_2} =  1151 + 480i + 240j
\]

Working with 100 digits of $5$-adic precision, we compute
\begin{align*}
J^+_{\psi_1, \psi_2} &= 50971141466526826096289662898361868496463698468806135561183036939036\\
&+ 9674029354607221223815165708202713711819464972332940921086896674730\alpha_5 + O(5^{97}),
\end{align*}
where $\alpha_5\in \CC_5$ satisfies $\alpha_5^2-\alpha_5-13=0$, the same polynomial that $\frac{1+\sqrt{53}}{2}$ satisfies.

The period $J^+_{\psi_1,\psi_2}$ satisfies, up to a root of unity, the polynomial
\[
41177889x^4 + 7867012x^3 + 33058502x^2 + 7867012x + 41177889.
\]
One checks that this gives an unramified extensions of the compositum of the fields of definition of the involved cycle and cocycle, as predicted by the conjecture. Also, note the factorization of the leading term of the minimal polynomial for $J^+_{\psi_1,\psi_2}$,
\[
41177889 = 3^4 \cdot 23^2 \cdot 31^2.
\]
The list of primes $3$, $23$ and $31$ appears in the first of the tables of the next subsection.
\subsection{Tables with quaternionic examples}
In this section we present the result of a batch of calculations with the quaternion algebras $B$ of discrimiants $D \in \{6, 10, 22\}$. These are all the cases --other than the modular curve cases-- for which the corresponding Shimura curve has genus $0$. In each case, we have chosen $p$ to be the smallest prime which is split in $B$: for $D=6$ we have taken $p=5$, and in the other two cases we have set $p=3$.

We consider the real quadratic fields of discriminant $<100$ for which $p$ is inert and the primes dividing $D$ are non-split. For each pair of \emph{different} fields $K_1$, $K_2$ (rather, to their maximal orders) we calculate $J_{\Delta_1,\Delta_2}^\star$, with $\star\in\{+,-,\text{even},\text{odd}\}$.

In order to recognize $J_{\Delta_1,\Delta_2}^\star$ algebraically, we take advantage of the fact that Conjecture~\ref{conj:main} predicts the field $H$ where $J_{\Delta_1,\Delta_2}^\star$ should belong. The support is also predicted in~\cite[\S 4]{darmon-vonk} and hence (by embedding $H^\times$ in $\QQ_{p^2}^\times$ and then taking $p$-adic logarithms) we can translate the problem of algebraically recognizing $J_{\Delta_1,\Delta_2}^\star$ into an additive lattice reduction problem.

In the row $\Delta_1$ and column $\Delta_2$ of each table we represent the result of this experiment corresponding to $J_{\Delta_1,\Delta_2}^\star$. A \textbf{?}-sign means that we were unsuccessful in recognizing the value algebraically. By the symbol \textbf{-} we mean that the quantity was recognized but it belongs to $\QQ(\sqrt{\Delta_1},\sqrt{\Delta_2})\subseteq H$ --we regard these quantities as ``trivial''--. In all other cases, the list of primes denotes the support of the recognized quantity: it is the set $S$ of rational primes $q$ defined by the property
\[
\forall\fq\subseteq\cO_H,\quad |J|_{\fq} \neq 0 \iff \fq\cap\ZZ \in S.
\]
If $S=\emptyset$ (which means that $J$ is a unit in $\cO_H$) then we write $1$ instead of leaving a blank space.

It is worth remarking that the plus and minus tables are symmetric, while this is not true for the even and odd ones. From this point of view, it seems justified to consider $\pm$-classes more ``primitive'' than the even and odd ones.
\begin{table}[h!]
\begin{mytable}{7}{Plus table}
 & 8 & 12 & 53 & 77 & 92 & 93 \\
  \midrule
8 &  &  -  &  -  & 3, 5 & 2, 3 & 5 \\
12 &  -  &  & 5 &  ?  & 2 &  -  \\
53 &  -  & 5 &  &  ?  & 3, 23, 31 & 2, 5, 41 \\
77 & 3, 5 &  ?  &  ?  &  &  ?  &  ?  \\
92 & 2, 3 & 2 & 3, 23, 31 &  ?  &  &  ?  \\
93 & 5 &  -  & 2, 5, 41 &  ?  &  ?  &  \\
\end{mytable}

\begin{mytable}{7}{Minus table}
 & 8 & 12 & 53 & 77 & 92 & 93 \\
  \midrule
8 &  &  1  &  -  & 3, 5 & 2, 3 & 2, 5 \\
12 &  1  &  &  2,5  &  ?  &  1  &  1  \\
53 &  -  & 2, 5 &  & 3, 5 & 2, 3, 23, 31 & 2, 5, 41 \\
77 & 3, 5 &  ?  & 3, 5 &  &  ?  &  ?  \\
92 & 2, 3 &  1  & 2, 3, 23, 31 &  ?  &  &  ?  \\
93 & 2, 5 &  1  & 2, 5, 41 &  ?  &  ?  &  \\
\end{mytable}
\caption{Tables for $D=6$, $p=5$, plus-minus classes.}
\end{table}

\begin{table}[h!]
\begin{mytable}{7}{Even table}
 & 8 & 12 & 53 & 77 & 92 & 93 \\
  \midrule
8 &  &  1  &  -  & 3 & 2, 3 & 2, 5 \\
12 &  1  &  &  ?  &  ?  & 2 &  1  \\
53 &  -  & 2, 5 &  &  ?  & 2, 3, 23 & 2, 41 \\
77 & 5 &  ?  &  ?  &  &  ?  &  ?  \\
92 & 2, 3 & 2 & 2, 31 &  ?  &  &  ?  \\
93 & 2, 5 &  1  & 2, 41 &  ?  &  ?  &  \\
\end{mytable}
\begin{mytable}{7}{Odd table}
 & 8 & 12 & 53 & 77 & 92 & 93 \\
  \midrule
8 &  &  1  &  -  & 5 & 2 & 2, 5 \\
12 &  1  &  &  ?  &  ?  & 2 &  1  \\
53 &  -  & 2, 5 &  &  ?  & 2, 31 & 2, 5 \\
77 & 3 &  ?  &  ?  &  &  ?  &  ?  \\
92 & 2 & 2 & 2, 3, 23 &  ?  &  &  ?  \\
93 & 2, 5 &  1  & 2, 5 &  ?  &  ?  &  \\
\end{mytable}
\caption{Tables for $D=6$, $p=5$, even-odd classes.}
\end{table}

\begin{table}[h!]
\begin{center}
\begin{mytable}{6}{Plus table}
 & 5 & 8 & 53 & 77 & 92 \\
 \midrule
5 &  &  -  &  -  & 3 & 2, 3 \\
8 &  -  &  &  -  & 3, 5 & 2, 3 \\
53 &  -  &  -  &  & 3, 5, 31 & 2, 3, 23, 31 \\
77 & 3 & 3, 5 & 3, 5, 31 &  &  ?  \\
92 & 2, 3 & 2, 3 & 2, 3, 23, 31 &  ?  &  \\
\end{mytable}
\begin{mytable}{6}{Minus table}
 & 5 & 8 & 53 & 77 & 92 \\
 \midrule
5 &  &  -  &  -  & 3 & 3 \\
8 &  -  &  &  -  & 3, 5 & 2, 3\\
53 &  -  &  -  &  &  ?  & 3, 23, 31 \\
77 & 3 & 3, 5 &  ?  &  &  ?  \\
92 & 3 & 2, 3 & 3, 23, 31 &  ?  &  \\
\end{mytable}
\caption{Tables for $D=10$, $p=3$, plus-minus classes.}
\end{center}
\end{table}

\begin{table}[h!]
\begin{center}
\begin{mytable}{6}{Even table}
 & 5 & 8 & 53 & 77 & 92 \\
5 &  &  -  &  -  & 3 & 2, 3 \\
8 &  -  &  &  -  & 5 & 2\\
53 &  -  &  -  &  &  ?  & 2, 31 \\
77 & 3 & 5 &  ?  &  &  ?  \\
92 & 2, 3 & 2, 3 & 2, 3, 23 &  ?  &  \\
\end{mytable}
\begin{mytable}{6}{Odd table}
 & 5 & 8 & 53 & 77 & 92 \\
5 &  &  -  &  -  &  1  & 2, 3 \\
8 &  -  &  &  -  & 3 & 2, 3\\
53 &  -  &  -  &  &  ?  & 2, 3, 23 \\
77 &  1  & 3 &  ?  &  &  ?  \\
92 & 2, 3 & 2 & 2, 31 &  ?  &  \\
\end{mytable}
\caption{Tables for $D=10$, $p=3$, even-odd classes.}
\end{center}
\end{table}

\begin{table}[h!]
\begin{center}
\begin{mytable}{5}{Plus table}
 & 8 & 29 & 44 & 77 \\
 8 &  &  -  &  -  & 2 \\
29 &  -  &  & 3 & 2 \\
44 &  -  & 3 &  & 2, 11 \\
77 & 2 & 2 & 2, 11 &  \\
\end{mytable}
\begin{mytable}{5}{Minus table}
 & 8 & 29 & 44 & 77 \\
 8 &  &  -  &  1  &  ?  \\
29 &  -  &  &  ?  &  ?  \\
44 &  1  &  ?  &  &  ?  \\
77 &  ?  &  ?  &  ?  &  \\
\end{mytable}
\caption{Tables for $D=22$, $p=3$, plus-minus classes.}
\end{center}
\end{table}

\begin{table}[h!]
\begin{center}
\begin{mytable}{5}{Even table}
 & 8 & 29 & 44 & 77 \\
 8 &  &  -  &  1  &  ?  \\
29 &  -  &  &  ?  &  ?  \\
44 &  1  &  ?  &  &  ?  \\
77 &  ?  &  ?  &  ?  &  \\
\end{mytable}
\begin{mytable}{5}{Odd table}
 & 8 & 29 & 44 & 77 \\
 8 &  &  -  &  1  &  ?  \\
29 &  -  &  &  ?  &  ?  \\
44 &  1  &  ?  &  &  ?  \\
77 &  ?  &  ?  &  ?  &  \\
\end{mytable}
\caption{Tables for $D=22$, $p=3$, even-odd classes.}
\end{center}
\end{table}

\subsection{Moving $p$ and $B$}\label{subsection:diferent ps}

Let $K_1=\QQ(\sqrt{53})$, which has narrow class number $1$, and let 
$K_2=\QQ(\sqrt{23})$, of narrow class number $2$.

Consider first the quaternion algebra $B_{10}=\left(\frac{2,-5}{\QQ}\right)$ of discriminant $10$. An embedding of the maximal order of $K_1$ in a maximal order $R$ of $B$ yields $\gamma_{\psi_1}= 51/2 + 49/2i + 21/2j$. Similarly, an embedding of the maximal order of $K_2$ in $R$ yields $\gamma_{\psi_2}=1151 - 720i + 240j + 240k$. Working with $200$ digits of $3$-adic precision, we compute
\begin{scriptsize}
\begin{align*}
J^\text{even}_{\psi_1,\psi_2} = 671432593119615754102633585711508084975279376970274924820959686886765982751024059399440196967 +\\ 854036156664899807234573316442426628332603932639256104698469526875913414165113678004656329424 \frac{1+\sqrt{53}}{2} + O(3^{195}).
\end{align*}
\end{scriptsize}

We repeat the calculation with the quaternion algebra $B_{6}=\left(\frac{2,3}{\QQ}\right)$ of discriminant $6$. The embeddings of $K_1$ and $K_2$ in $B_{6}$ this time give rise to $\gamma_{\psi_1}' = 51/2 + 21/2i + 7/2j $ and $\gamma_{\psi_2}'=  1151 + 480i + 240j$, respectively. Working with $200$ digits of $5$-adic precision we compute
\begin{scriptsize}
\begin{align*}
J^\text{even}_{\psi_1',\psi_2'} =
 2235158967056605935022903962318227997528221577810184623092802398073249517900702179841205111669\\29820333143729033784117048388613175877216081 + 1888129453960046774715006105424989658239316060832855\\95774155493419628402609686469558179779684297934788222122759958929106946068250088369040\frac{1+\sqrt{53}}{2} + O(5^{197}).
\end{align*}
\end{scriptsize}

Consider $M$ the field generated by a root of the polynomial $x^8 - 4x^7 + 84x^6 - 238x^5 + 1869x^4 - 3346x^3 + 7260x^2 - 5626x + 3497$, and note that $M$ contains $\QQ(\sqrt{53},\sqrt{23})$. Also, $M$ embeds in both $\QQ_{3^2}$ and $\QQ_{5^2}$ via $\iota_3$ and $\iota_5$, respectively. We check that there exists $\alpha\in M$ (supported only on primes above $2$, $3$, $5$, $23$ and $31$) and units $u_1$, $u_2$ in $\QQ(\sqrt{53},\sqrt{23})$ satisfying
\[
\iota_3(\alpha u_1)=J^\text{even}_{\psi_1,\psi_2},\text{ and } \iota_5(\alpha u_2) = 
J^\text{even}_{\psi_1',\psi_2'}.
\]
In fact, the units $u_1$ and $u_2$ belong to the rang $2$ subgroup generated by the fundamental units of $K_1$ and $K_2$. It is worth noting that the primes appearing in the support of $\alpha$ are predicted by the Gross--Zagier factorization, as explained in~\cite[\S 4]{darmon-vonk}

\subsection{Large genus example}
Consider the quaternion algebra $B/\QQ$ of discriminant $15$ given by $B=\left(\frac{2,15}{\QQ}\right)$, and let $p=7$. We consider the maximal order
$R = \langle 1, i, \frac{1 + i + j}{2}, \frac{2-i+k}{4} \rangle$.
 Consider the quadratic field $K_1=\QQ(\sqrt{17})$ and the embedding
\[
\psi_1\colon \cO_{K_1} \injects R,\quad \frac{1+\sqrt{17}}{2}\mapsto 1/2 - 1/2i - 1/2j.
\]
This yields the hyperbolic element $\gamma_{\psi_1}$:
\[
\gamma_1 =  33 + 8i + 8j.
\]
Similarly, we consider $K_2=\QQ(\sqrt{33})$, and the embedding
\[
\psi_2\colon\cO_{K_2}\injects R,\quad \frac{1+\sqrt{33}}{2}\mapsto  1/2 - 9/4i - j - 3/4k.
\]
This yields the hyperbolic element $\gamma_{\psi_2}$:
\[
\gamma_{\psi_2} = 1057 + 828i + 368j + 276k.
\]

In order to lift the classes, we needed to act by $T_2 + 1$, which kills the cuspidal space for $\Gamma_0(15)$. Working with $100$ digits of $7$-adic precision, we compute
{\scriptsize
\begin{align*}
J^+_{\psi_1, \psi_2} &= 64094024229051011328172608155448301705695076718018850887636370131427738532434310728\\
&+ 47794262697757308857586073314123359973162778051638804058607411482263821225069561350\alpha_7 + O(7^{98}), \end{align*}
}
where $\alpha_7\in\CC_7$ satisfies $\alpha_7^2 - \alpha_7 - 4=0$. The period $J^+_{\psi_1,\psi_2}$ satisfies, up to a root of unity, the polynomial
\[
680625x^4 - 2444871x^3 + 3533392x^2 - 2444871x + 680625.
\]
This element has support at the rational primes $3$, $5$ and $11$, and it does indeed generate the predicted compositum of $\QQ(\sqrt{17})$ (of trivial narrow class group) with the narrow class field of $\QQ(\sqrt{33})$.

\subsection{An example over a real quadratic field}\label{ex:real quadratic F}

Consider $F=\QQ(\sqrt{5})$, with ring of integers $\ZZ[w]$, $w=\frac{1+\sqrt{5}}{2}$. Consider the quaternion algebra $B/F$ of discriminant $(2)$ given by $B=\left(\frac{-w,-2}{F}\right)$. In this example we take $\fp = (-3w+2)$, an ideal of norm $11$. We consider the maximal order
\[
R = \langle 1, i, 2w + 2i + j, 2w+2 + 2wi + k \rangle.
\]
We consider the quadratic extension $K_1/F$ given by adjoining to $F$ the square root $w_1$ of $1-2w$, and the embedding
\[
\psi_1\colon \cO_{K_1}\injects R,\quad w_1\mapsto  (w-2)i - j.
\]
This yields the hyperbolic element $\gamma_{\psi_1}$:
\[
\gamma_{\psi_1} =  w - 2 + (2w - 3)i + (w - 1)j
\]
Similarly, we consider the quadratic extension $K_2/F$ given by adjoining to $F$ the square root $w_2$ of $9-14w$, and the embedding
\[
\psi\colon \cO_K\injects R,\quad w_2\mapsto  (-3w + 2)i + (w - 2)k
\]
This yields the hyperbolic element $\gamma_{\psi_2}$:
\[
\gamma_{\psi_2} = -55w + 88 + (-50w + 81)i + (34w - 55)k
\]

Working with 60 digits of $\fp$-adic precision, we compute
\[
J^+_{\psi_1,\psi_2} = 2650833861085011569846208847449970229624664608755690791954838 + O(11^{59}),
\]
which satisfies the polynomial
\[
25420x^4 - 227820x^3 + 2200011x^2 - 27566220x + 372174220.
\]
One readily checks that this gives an unramified extension of the compositum of the fields of definition of the involved cycle and cocycle, as predicted by the conjecture.

\bibliographystyle{amsalpha}
\bibliography{refs}

\newcommand{\etalchar}[1]{$^{#1}$}
\providecommand{\bysame}{\leavevmode\hbox to3em{\hrulefill}\thinspace}
\providecommand{\MR}{\relax\ifhmode\unskip\space\fi MR }
\providecommand{\MRhref}[2]{%
  \href{http://www.ams.org/mathscinet-getitem?mr=#1}{#2}
}
\providecommand{\href}[2]{#2}
\begin{thebibliography}{{Geh}20}

\bibitem[DV20a]{dv2}
H.~Darmon and J.~Vonk, \emph{Arithmetic intersections of modular geodesics},
  preprint (2020).

\bibitem[DV20b]{darmon-vonk}
\bysame, \emph{Singular moduli for real quadratic fields: a rigid analytic
  approach}, Accepted for publication in the Duke Mathematical Journal (2020).

\bibitem[{Geh}20]{lennart}
Lennart {Gehrmann}, \emph{{On quaternionic rigid meromorphic cocyles}}, arXiv
  e-prints (2020), arXiv:2009.04957.

\bibitem[GM14]{shpquat}
X.~Guitart and M.~Masdeu, \emph{Overconvergent cohomology and quaternionic
  {D}armon points}, J. Lond. Math. Soc. (2) \textbf{90} (2014), no.~2,
  495--524. \MR{3263962}

\bibitem[GMS15]{gms1}
Xavier Guitart, Marc Masdeu, and Mehmet~Haluk Seng\"{u}n, \emph{Darmon points
  on elliptic curves over number fields of arbitrary signature}, Proc. Lond.
  Math. Soc. (3) \textbf{111} (2015), no.~2, 484--518. \MR{3384519}

\bibitem[Gre09]{Gr}
M.~Greenberg, \emph{{S}tark-{H}eegner points and the cohomology of quaternionic
  {S}himura varieties}, Duke Math. J. \textbf{147} (2009), no.~3, 541--575.
  \MR{2510743 (2010f:11097)}

\bibitem[S{\etalchar{+}}20]{sage}
W.\thinspace{}A. Stein et~al., \emph{{S}age {M}athematics {S}oftware ({V}ersion
  9.1)}, The Sage Development Team, 2020, {\tt http://www.sagemath.org}.

\bibitem[Voi14]{voight-quaternion}
John Voight, \emph{The arithmetic of quaternion algebras}, preprint (2014).

\end{thebibliography}
\end{document}